\documentclass[11pt]{amsart}
\usepackage{mathrsfs}
\usepackage{amsfonts}
\usepackage{amsmath}
\usepackage{amssymb,latexsym}
\usepackage{graphicx}
\usepackage{amsthm}
\usepackage{hyperref}
\usepackage{indentfirst}
\usepackage{CJK}
\usepackage{epsfig,graphicx,picins,picinpar,subfigure}
\usepackage{pstricks}
\usepackage{color,soul}
\usepackage{fancyvrb}
\usepackage[numbers,sort&compress]{natbib}

\def\be{\begin{equation}\begin{cases}\begin{aligned}}
\def\ee{\end{aligned}\end{cases}\end{equation}}
\def\ba{\begin{equation*}\begin{aligned}}
\def\ea{\end{aligned}\end{equation*}}
\def\bs{\begin{equation}\begin{aligned}}
\def\es{\end{aligned}\end{equation}}
\def\bp{\begin{prop}}
\def\ep{\end{prop}}

\normalsize \makeatletter
\usepackage{xcolor}

\newtheorem{theorem}{Theorem}[section]

\newtheorem{prop}[theorem]{Proposition}

\theoremstyle{remark}

\begin{document}

\title[Gradient estimate of the solutions to Hessian equations]
{Gradient estimate of the solutions to Hessian equations with oblique boundary value}


\begin{abstract}
In this paper, we study  Hessian equations with prescribed contact angle boundary value or oblique derivative boundary value and finally derive the a priori global gradient estimate for the admissible solutions.
\end{abstract}

\author{PeiHe Wang}
\address{School of Mathematical Sciences,
 Qufu Normal University,
      Qufu, 273165, Shandong Province, China}
\email{peihewang@hotmail.com}

\date{}
\maketitle

\section{Introduction}

In this paper, we consider the following Hessian equation with oblique boundary value,
\begin{equation}\label{summary}
   \begin{split}
   \left\{\aligned
&\sigma_{k}(u_{ij})=f(x,u)  &\mathrm{\mbox {in}} & \ \Omega,\\
&G(x,\, D u)=0 &\mathrm{\mbox {on}}& \   \partial\Omega,\\
\endaligned\right.
   \end{split}
\end{equation}
where $\Omega$ is a bounded domain in $\mathrm{R}^n$ with smooth boundary, $f(x,\,t)$ and $G(x,\, \overrightarrow{p})$ are smooth functions defined respectively on $\Omega\times \mathbf{R}$ and $\overline{\Omega}\times \mathbf{R}^n$. We mainly study two general but important cases of $G(x,\, Du)$, one is the prescribed contact angle boundary value problem and the other is the oblique derivative boundary value problem. The topic in this paper is also concentrated on the global gradient estimate which would be one step forward to conclude the existence of the solution to the problem \eqref{summary}.

Hessian equations including Laplace equations and Monge-Amp$\mathrm{\grave{e}}$re equations as their special cases with various boundary values are in no doubt an interesting subject in recent years, many topics in differential geometry, convex geometry and optimal transport etc have close relations with this kind of elliptic equations. For the given boundary value, one may firstly be interested in the existence of the solution. In general, it is necessary to get the $C^{2,\alpha}$ estimate to conclude the existence of the solution. For instance, when the boundary value is of the Dirichlet type, one can refer to \cite{XJW1}, \cite{CW} and \cite{CNSIII} for the existence results. For the Neumann boundary value, Trudinger \cite{Trudinger} considered the special domain case and got the existence result. Also, he conjectured in \cite{Trudinger} that one can solve the problem in sufficiently smooth uniformly convex domains. Recently, Ma-Qiu \cite{MQ} gave a positive answer to this problem and solved the the Neumann problem of $k-$Hessian equations in uniformly convex domains. Chen and Zhang \cite{CZh} considered the Hessian quotient equation and also derived the existence results with Neumann boundary condition.

Now, it is of natural interest to consider the existence of the solutions to Hessian equations with the other types of boundary value problems such as prescribed contact angle boundary value and oblique derivative boundary value. It seems to be a little more complicated for these kinds of boundary values. For instance, a necessary condition for the existence of the solution to Monge-Amp$\mathrm{\grave{e}}$re equations was exhibited in \cite{MaPro} and \cite{JUrbas1}. Till now, there are only a few progress results on this topic. In \cite{XJW2}, the oblique derivative boundary problems for Monge-Amp$\mathrm{\grave{e}}$re equations were considered and the existence of the solutions to two dimension Monge-Amp$\mathrm{\grave{e}}$re equations was derived, and the generalized solutions for general dimension Monge-Amp$\mathrm{\grave{e}}$re equations were also considered. In \cite{JUrbas2}, \cite{JUrbas3} and \cite{JUrbas4}, Urbas also derived some existence results for Monge-Amp$\mathrm{\grave{e}}$re equations with oblique  derivative boundary value. For some augmented Hessian equations with oblique boundary value, Jiang and Trudinger in \cite{JT1} and \cite{JT2} considered the existence result. Wang \cite{XJW} derived the interior gradient estimate of the solutions to $k-$curvature equations and Deng and Ma \cite{DM} got the global gradient estimate for $k-$curvature equations with prescribed contact angle boundary value. It is still open for the existence of the solutions to $k-$curvature equations and Hessian equations with prescribed contact angle or oblique derivative boundary value. In this paper, we make an attempt for this problem and finally will derive the global gradient estimate for admissible solutions to Hessian equations with these kinds of boundary conditions which would be considered as a little step forward to the existence of the solutions to these interesting problems.

Gradient estimate of the solutions to various partial differential equations is an important and interesting issue in the study of P.D.E.  Usually, it includes interior gradient estimate and global gradient estimate which respectively have close relation to Liouville type results and the existence of the solution to P.D.E. One can refer to \cite{XJW1}, \cite{CNSIII}, \cite{XJW}, \cite{XJW2}, \cite{CZh}, \cite{chen},  \cite{CXZ}, \cite{Lieb},  \cite{MWW},   \cite{MQX},  \cite{WZ}, \cite{DM} etc and the references therein for more details.

The rest of the paper is organized as follows. In section 2, we introduce some notations and preliminaries for the proceed of the paper. In Section 3, we give the global gradient estimate of the solution for Hessian equations with prescribed contact angle boundary value, and in Section 4 we come to deal with the oblique derivative boundary data case.

\section{Notations and Preliminaries}

 In this section, we list some notations and and preliminaries which are necessary for the gradient estimate.

Firstly, we denoted by $d(x)=\mathrm{dist}(x,\, \partial \Omega)$ the distance from $x$ to $\partial \Omega$, the boundary of a bounded smooth domain $\Omega$. As a known fact, $d(x)$ is also smooth near the boundary, such as on the annular domain $\Omega_{\mu_1}=\{x\in \Omega\,|\, d(x)\le \mu_1 \}$,  where $\mu_1$ is a positive constant related to the domain.

 Secondly,  we give some basic properties of elementary symmetric functions, denoted by $\sigma_k(\lambda)$ for $\lambda\in \mathbf{R}^n$, which could be found in \cite{XJW1} and \cite{CNSIII}.

We denoted by $\sigma_k(\lambda|i)$ the $k-$th symmetric function with $\lambda_i=0$ and $\sigma_k(\lambda|ij)$ the $k-$th symmetric function with $\lambda_i=\lambda_j=0.$  Then we have the following propositions.

\begin{prop}\label{symmetric function}
Assume $\lambda=(\lambda_1,\,\lambda_2,\cdots,\lambda_n)\in \mathbf{R}^n$, and $k=1,\,2,\cdots,n,$ then we have
\begin{equation}\label{symmetric}
   \begin{split}
   &\sigma_k(\lambda)=\sigma_k(\lambda|i)+\lambda_i\sigma_{k-1}(\lambda|i),\ \  1\le i\ \le n,\\
   & \sum_{i=1}^n \lambda_i\sigma_{k-1}(\lambda|i)=k\sigma_k(\lambda),\\
   &  \sum_{i=1}^n \sigma_{k}(\lambda|i)=(n-k)\sigma_k(\lambda).
   \end{split}
\end{equation}

\end{prop}

Recall that the Garding's cone is defined as
$$\Gamma_k=\{\lambda\in  \mathbf{R}^n \,|\, \sigma_i >0,\ \forall 1\le i\le k \}.$$

\begin{prop}\label{Elliptic}
Assume $k\in \{1,\,2,\cdots,n\}$ and $\lambda\in \Gamma_k$, suppose that
$$\lambda_1\ge\cdots\ge  \lambda_k\ge \cdots \ge \lambda_n, $$
then we have
$$\sigma_{k-1}(\lambda|n)\ge \cdots\ge \sigma_{k-1}(\lambda|k)\ge \cdots \ge \sigma_{k-1}(\lambda|1)>0$$
and
\begin{equation}\label{Fkk-F}
   \begin{split}
  \sigma_{k-1}(\lambda|k)\ge C(n,\,k) \sum_{i=1}^n \sigma_{k-1}(\lambda|i)\,.
   \end{split}
\end{equation}
\end{prop}

Remark that if the eigenvalues of $(u_{ij})$, denoted also by $(\lambda_1,\,\lambda_2,\cdots,\lambda_n)$, are located in $\Gamma_k$, then the equation in \eqref{summary} is elliptic and we will call this kind of solution as ``$k-$admissible" solution.

We also list the generalized Newton-MacLaurin inequality in the following which includes the Newton inequality and the MacLaurin inequality as the special cases.
\begin{prop}\label{NewtonInequality}
Assume $\lambda\in \Gamma_k$, and $k,\,l,\,r,\,s\in \{0,\,1,\,2,\cdots,n\}$ with $k> l\ge0,\ r>s\ge0$, $k\ge r,\ l\ge s$ and $k-l\ge r-s$, we have
\begin{equation}\label{Newtoninequality}
   \begin{split}
   \left[ \frac{\frac{\sigma_k(\lambda)}{C^k_n}}{\frac{\sigma_l(\lambda)}{C^l_n}}\right]^{\frac{1}{k-l}}\le \left[ \frac{\frac{\sigma_r(\lambda)}{C^r_n}}{\frac{\sigma_s(\lambda)}{C^s_n}}\right]^{\frac{1}{r-s}},
   \end{split}
\end{equation}
and the equality holds if and only if $\lambda_1=\lambda_2=\cdots=\lambda_n>0.$
\end{prop}
As the last point of this section, we also state that the universal constant $C$ during the whole paper may change from line to line.

\section{Prescribed contact angle boundary data}

In this section, we set out to get the gradient estimate of the admissible solution to Hessian equations with prescribed contact angle boundary value. In a word, we will prove the following theorem.

\begin{theorem}\label{PCAB}
Let $\Omega$ be a smooth bounded domain in $\mathbf{R}^n(n\ge 2)$ and $u$ be the admissible solution to the following Hessian equations with prescribed contact angle boundary value,
\begin{equation}\label{Pangle}
   \begin{split}
   \left\{\aligned
&\sigma_{k}(u_{ij})=f(x,u)  &\mathrm{\mbox {in}} & \ \Omega,\\
&\frac{\partial u}{\partial\nu}=-\cos\theta\sqrt{1+|Du|^{2}} &\mathrm{\mbox {on}}& \   \partial\Omega.\\
\endaligned\right.
   \end{split}
\end{equation}
Assume that $f(x,\,t)$ is a positive smooth function defined on $\Omega \times \mathbf{R}$ with $f_t \ge 0$ and $\theta(x)$ is a smooth function defined on $\overline{\Omega}$ with $ |\cos\theta|\leq 1-b<1$ for some positive constant $b$.  $\nu$~ is denoted to be the inward unit normal along ~$\partial \Omega$~. Also we assume that we have already got the ~$C^{0}$~ estimate as ~$|u|\leq M$~. Then, there exists a positive constant $C=C(M,\, n,\, \Omega,\, b,\, |\theta|_{C^2(\overline{\Omega})} ,\, |f|_{C^1(\Omega\times[-M,\ M])})$ such that
\begin{equation}
\begin{aligned}
|Du|\le C.
\end{aligned}
\end{equation}

\end{theorem}

\begin{proof}

Due to \cite{XJW1}, we have already known the interior gradient estimate, so we only need to get the gradient estimate near boundary, denoted by ~$\Omega_{\mu}$~, where ~$\mu\le \mu_1$~ is a positive constant to be determined later.

Let $v=\sqrt{1+|Du|^{2}}$,\ $w=v+\sum\limits_{l=1}^nu_{l}d_l\cos\theta$ and let $h(t),\ \tau$  be a smooth function and a positive constant respectively to be determined later. We choose the auxiliary function
\begin{equation*}
\begin{aligned}
\Phi=\log w+h(u)+\tau d\,.
\end{aligned}
\end{equation*}
Assume $\Phi$~achieves its maximum on the domain $\overline{\Omega_{\mu}}$ at the point~$x_{0}$, according to the interior gradient estimate, we can only consider  the following two cases.

{\bf Case I:~$x_{0}\in\partial\Omega$.}

For convenience, we choose a coordinate around $x_0$ such that $\nu=\frac{\partial}{\partial x_n}$, assume $\frac{\partial}{\partial x_i}(i=1,2,\cdots,n-1)$ are tangent to $\partial \Omega$. Under this coordinate, we have

$$\frac{{\partial d}}{{\partial {x_i}}} = 0,\  \frac{{\partial d}}{{\partial {x_n}}} = 1,\   \frac{{{\partial ^2}d}}{{\partial {x_n}\partial {x_\alpha }}} = 0,\   \frac{{{\partial ^2}d}}{{\partial {x_i}\partial {x_j}}} =  - {\kappa _i}{\delta _{ij}}, $$
 where $1\le i,j<n-1$, $1\le \alpha \le n$ and $\kappa_i(i=1,2,\cdots,n-1)$ are the principal curvatures of $\partial \Omega$ at $x_0$.

By the fact that $x_0$ is the maximum point on the boundary, we have
   \begin{equation}\label{Phi-i}
\begin{split}
   0 =& \Phi_i = \frac{w_i}{w} + h'u_i + \alpha d_i = \frac{w_i}{w } + h'u_i, \ \ i=1,2,\cdots,n-1\,,
     \end{split}
\end{equation}
  and
  \begin{equation}\label{Phi-n}
\begin{split}
   0 \ge& \Phi_n = \frac{w_n}{w} + h'u_n + \alpha d_n = \frac{w_n}{w} + h'u_n + \alpha.
     \end{split}
\end{equation}

   By a direct computation, we have
   \begin{equation}\label{eta-n}
\begin{split}
w_n =& v_n + u_{n n} \cos\theta + u_n (\cos \theta)_n   \\
   =& \frac{\sum\limits_{\alpha = 1}^n u_\alpha u_{\alpha n}}{v} + u_{nn}\cos \theta  + u_n (\cos \theta )_n \\
   =& \frac{\sum\limits_{i = 1}^{n - 1} u_i u_{in}}{v} + \frac{ u_n u_{nn}}{v} + u_{nn}\cos \theta  + u_n (\cos \theta )_n \\
      =& \frac{\sum\limits_{i = 1}^{n - 1} u_i u_{ni}}{v} + \frac{\sum\limits_{i,j = 1}^{n - 1} u_i k_{ij} u_j}{v} + u_n (\cos \theta )_n,
\end{split}
\end{equation}
where we denote by  $k_{ij}$ the Weingarten matrix of the boundary with respect to $\nu$\,.

   Differentiating $u_n $ along $\partial \Omega $, we obtain for $i=1,2,\cdots,n-1$ that
   \ba
    u_{ni} =& (- v\cos \theta)_i = -{v_i}\cos \theta - v(\cos \theta )_i\\
    =& -(w_i -u_{ni}\cos \theta-\sum_{l=1}^{n-1}u_ld_{li}\cos\theta-u_n(\cos \theta )_i)\cos \theta - v(\cos \theta )_i \\
   =& -w_i \cos \theta  + u_{ni}\cos^2\theta  +\sum_{l=1}^{n-1}u_ld_{li}\cos^2\theta+ u_n \cos \theta (\cos \theta )_i - v(\cos \theta )_i\,,
   \ea
   furthermore, using \eqref{Phi-i}
      we can get
   \begin{equation}\label{wni}
\begin{split}
u_{ni} = \frac{ h'w{u_i}\cos \theta+\sum\limits_{l=1}^{n-1}u_ld_{li}\cos^2\theta  - v(1 + \cos^2 \theta )(\cos \theta )_i}{{\sin^2}\theta }.
\end{split}
\end{equation}
   Plug \eqref{wni} into \eqref{eta-n} we then have
    \ba
   w_n =& \frac{\sum\limits_{i = 1}^{n - 1} u_i u_{ni}}{v} + \frac{\sum\limits_{i,j = 1}^{n - 1} u_i k_{ij} u_j }{v} + u_n (\cos \theta )_n\\
   =&\frac{\sum\limits_{i = 1}^{n - 1} u_i\left[ h'w{w_i}\cos \theta +\sum\limits_{l=1}^{n-1}u_ld_{li}\cos^2\theta  - v(1 + \cos^2 \theta )(\cos \theta )_i\right]}{v\sin^2 \theta } \\
   &+ \frac{\sum\limits_{i,j = 1}^{n - 1} u_i k_{ij} u_j}{v} + u_n(\cos \theta )_n.
   \ea
   Then
   \ba 0 \ge\Phi_n=& \frac{\sum\limits_{i = 1}^{n - 1} u_i [ h'w{u_i}\cos \theta +\sum\limits_{l=1}^{n-1}u_ld_{li}\cos^2\theta  - v(1 + \cos^2 \theta )(\cos \theta )_i]}{w v{\sin^2}\theta } \\
   &+ \frac{\sum\limits_{i,j = 1}^{n - 1} u_i k_{ij} u_j }{w v}   + \frac{ u_n(\cos \theta )_n}{w} + h'{u_n} + \alpha \\
    =& \frac{h'\cos \theta \sum\limits_{i = 1}^{n - 1}u_i^2 }{v\sin^2\theta} - \frac{(1 + \cos^2\theta )\sum\limits_{i = 1}^{n - 1}u_i (\cos \theta )_i}{w \sin^2 \theta }+\frac{\sum\limits_{i,l=1}^{n-1}u_ld_{li}u_i\cos^2\theta }{w v{\sin^2}\theta}\\
     &+ \frac{\sum\limits_{i,j = 1}^{n - 1} u_i  k_{ij} u_j }{w v}+ \frac{ u_n (\cos \theta )_n}{w} + h'{u_n} + \alpha \\
   =& \frac{h'\cos \theta (v^2\sin^2\theta - 1)}{v\sin^2\theta} - \frac{(1 + \cos^2\theta )\sum\limits_{i = 1}^{n - 1}u_i (\cos \theta )_i}{w \sin^2 \theta }+\frac{\sum\limits_{i,l=1}^{n-1}u_ld_{li}u_i\cos^2\theta }{w v{\sin^2}\theta}\\
     &+ \frac{\sum\limits_{i,j = 1}^{n - 1} u_i  k_{ij} u_j }{w v}+ \frac{ u_n (\cos \theta )_n}{w} + h'{u_n} + \alpha \\
   =& - \frac{h'\cos \theta}{v\sin^2\theta} -  \frac{(1 + \cos^2\theta )\sum\limits_{i = 1}^{n - 1}u_i(\cos \theta )_i }{w \sin^2\theta} +\frac{\sum\limits_{i,l=1}^{n-1}u_ld_{li}u_i\cos^2\theta }{w v{\sin^2}\theta}\\
   &+ \frac{\sum\limits_{i,j = 1}^{n - 1} u_i  k_{ij}u_j }{w v}+ \frac{ u_n (\cos \theta )_n}{w} + \alpha .
   \ea
   Without loss of generality, we may assume that $v$  is large such that if  $\alpha $ is chosen large enough determined by $\theta $ and the geometry of  $\partial \Omega$, the right hand of the above inequality will be positive which shows that this case will not occur at all.

{\bf Case II:~$x_{0}\in\Omega_{\mu}$\,.}

At this point, we can assume that $|Du|$ is large enough such that $|Du|, w, v$ are equivalent with each other. Remark that the Einstein summation convention will be adopted during all the calculations if no otherwise specified.

Since $x_{0}$ is the maximum point, we then have
\begin{equation*}
\begin{aligned}
\Phi_{i}=\frac{w_{i}}{w}+h^{'}u_{i}+\tau d_{i},
\end{aligned}
\end{equation*}
it follows that
\begin{equation}\label{1-order-1}
\begin{aligned}
w_{i}=-w(h^{'}u_{i}+\tau d_{i})\,.
\end{aligned}
\end{equation}
By the definition of ~$w$~, we have
\begin{equation*}
\begin{aligned}
w_{i}&=\frac{u_{l}u_{li}}{v}+u_{li}d_{l}\cos \theta+u_{l}d_{li}\cos \theta+u_{l}d_{l}(\cos \theta)_{i}\\
&=(\frac{u_{l}}{v}+d_{l}\cos \theta)u_{li}+u_{l}d_{li}\cos \theta+u_{l}d_{l}(\cos \theta)_{i}.
\end{aligned}
\end{equation*}
Therefore,
\begin{equation}\label{1-order}
\begin{aligned}
-w(h^{'}u_{i}+\tau d_{i})=(\frac{u_{l}}{v}+\cos \theta d_{l})u_{li}+u_{l}d_{li}\cos \theta+u_{l}d_{l}(\cos \theta)_{i}\,.
\end{aligned}
\end{equation}

We now come to deal with $\Phi_{ij}$. By \eqref{1-order-1} we derive that
\begin{equation}
\begin{aligned}
\nonumber\Phi_{ij}&=\frac{w_{ij}}{w}-\frac{w_{i}w_{j}}{w^{2}}+h'u_{ij}+h'' u_{i}u_{j}+\tau d_{ij}\\
&=\frac{w_{ij}}{w}-(h' u_{i}+\tau d_{i})(h' u_{j}+\tau d_{j})+h' u_{ij}+h'' u_{i}u_{j}+\tau d_{ij}\\
&=\frac{w_{ij}}{w}-\tau h' u_{i}d_{j}-\tau h' u_{j}d_{i}-\tau^{2}d_{i}d_{j}+h' u_{ij}+[h''-(h')^{2}]u_{i}u_{j}+\tau d_{ij}\,.
\end{aligned}
\end{equation}

Following \cite{DM}, we take the coordinate around $x_{0}$ such that $(u_{ij}) $ is diagonal at this point and all the following calculation will be done at this point. Denoted by $F^{ij}$ the derivative $\frac{\partial \sigma_{k}(u_{ij})}{\partial u_{ij}}$ and  $F$ the sum $\sum_{i=1}^n F^{ii}.$ We then have
\begin{equation}\label{111}
\begin{aligned}
0\ge F^{ij}\Phi_{ij}&=\frac{F^{ij}w_{ij}}{w}+[h''-(h')^{2}]F^{ij}u_{i}u_{j}+h' F^{ij}u_{ij}\\
&\quad+\tau F^{ij}d_{ij}-2\tau h' F^{ij}u_{i}d_{j}-\tau^{2}F^{ij}d_{i}d_{j}.\\
&=\uppercase\expandafter{\romannumeral1}+\uppercase\expandafter{\romannumeral2}+\uppercase\expandafter{\romannumeral3}\,,
\end{aligned}
\end{equation}
where
\begin{equation*}
\begin{aligned}
\uppercase\expandafter{\romannumeral1}&=\frac{F^{ij}w_{ij}}{w}\,,\\
\uppercase\expandafter{\romannumeral2}&=[h''-(h')^{2}]F^{ij}u_{i}u_{j}\,,\\
\uppercase\expandafter{\romannumeral3}&=h' F^{ij}u_{ij}+\tau F^{ij}d_{ij}-2\tau h' F^{ij}u_{i}d_{j}-\tau^{2}F^{ij}d_{i}d_{j}\,.
\end{aligned}
\end{equation*}

For the last term, we can easily have
\begin{equation}\label{III}
\begin{aligned}
\uppercase\expandafter{\romannumeral3}=h' F^{ij}u_{ij}+\tau F^{ij}d_{ij}-2\tau h' F^{ij}u_{i}d_{j}-\tau^{2}F^{ij}d_{i}d_{j}\geq -C |Du| F\,.
\end{aligned}
\end{equation}

In the following, we come to deal with the first term $\uppercase\expandafter{\romannumeral1}$. The key point is to calculate $F^{ij}w_{ij}$. By a direct calculation, we can deduce that
\begin{equation}
\begin{aligned}
\nonumber w_{ij}&=(\frac{u_{l}}{v}+d_{l}\cos \theta)u_{lij}+(\frac{u_{l}}{v}+d_{l}\cos \theta)_{j}u_{li}\\
&\quad +u_{lj}d_{li}\cos \theta+u_{l}(d_{li}\cos \theta)_{j}+u_{lj}d_{l}(\cos \theta)_{i}+u_{l}(d_{l}(\cos \theta)_{i})_{j}\\
&=(\frac{u_{l}}{v}+d_{l}\cos \theta)u_{lij}+(\frac{u_{lj}}{v}-\frac{u_{l}u_{k}u_{kj}}{v^{3}})u_{li}+( d_{l}\cos \theta)_{j}u_{li}\\
&\quad+u_{lj}d_{li}\cos \theta +u_{l}(d_{li}\cos \theta)_{j}+u_{lj}d_{l}(\cos \theta)_{i}+u_{l}(d_{l}(\cos \theta)_{i})_{j}.
\end{aligned}
\end{equation}
Hence,
\begin{equation}
\begin{aligned}
F^{ij}w_{ij}&=(\frac{u_{l}}{v}+d_{l}\cos \theta)D_{l}f+(\frac{1}{v}-\frac{u_{i}^{2}}{v^{3}})F^{ii}u_{ii}^{2}+(d_{i}\cos \theta)_{i}F^{ii}u_{ii}\\
&\quad+F^{ii}u_{ii}d_{ii}\cos \theta +F^{ij}u_{l}(d_{li}\cos \theta)_{j}+F^{ii}u_{ii}d_{i}(\cos \theta)_{i}+F^{ij}u_{l}(d_{l}(\cos \theta)_{i})_{j}\\
&\geq (\frac{1}{v}-\frac{u_{i}^{2}}{v^{3}})F^{ii}u_{ii}^{2}+2(d_{i}\cos \theta)_{i}F^{ii}u_{ii}-C |Du| F-C|Du|.
\end{aligned}
\end{equation}

For the choice of the coordinate and \eqref{1-order}, we have at $x_0$ that
\begin{equation}
\begin{aligned}
 -w(h' u_{i}+\tau d_{i})=(\frac{u_{i}}{v}+ d_{i}\cos \theta)u_{ii}+u_{l}(d_{l}\cos \theta )_{i}\,,\  \ i=1, 2, \cdots, n.
\end{aligned}
\end{equation}

Setting
\begin{equation}
\nonumber K=\{ i\in \mathrm{I}\ \mid\ |d_{i}\cos \theta|+\frac{b}{8n}\leq |\frac{u_{i}}{v}| \},
\end{equation}
where\ $\mathrm{I}=\{1, 2, \cdot\cdot\cdot, n\}$. It is obvious that the index set $K$ is not empty and if we further assume that $v$ is large enough, we can assume that
$$|\tau d_{i}|\leq \frac{1}{2} h' |u_{i}|,\ |u_{l}(d_{l}\cos \theta)_{i}| \leq \frac{1}{4} | h' w u_{i}| \ \ \mbox{for}\ \  i\in K.$$
Note that we here need $h'$  have a positive bound which will be satisfied later. Under these assumptions, we have
\begin{equation}\label{uiiestimate}
  \begin{split}
-Ch'w|u_i|\le u_{ii}\leq 0 \ \ \mbox{for}\ \  i\in K.
  \end{split}
\end{equation}
Then for $i\in K$, we have by \eqref{Fkk-F} that
$$
 F^{ii}\geq F^{kk}\geq CF\,.
$$

Hence,
\begin{equation}
\begin{aligned}
F^{ij}w_{ij}\geq & \sum\limits_{i=1}^{n}\left( (\frac{1}{v}-\frac{u_{i}^{2}}{v^{3}})F^{ii}u_{ii}^{2}-2(d_{i}\cos \theta)_{i}F^{ii}u_{ii}\right )-C|Du|F-C|Du|\\
=&\sum\limits_{i\in K}\left((\frac{1}{v}-\frac{u_{i}^{2}}{v^{3}})F^{ii}u_{ii}^{2}-2( d_{i}\cos \theta)_{i}F^{ii}u_{ii}\right )\\
&+\sum\limits_{i\notin K}\left((\frac{1}{v}-\frac{u_{i}^{2}}{v^{3}})F^{ii}u_{ii}^{2}-2(d_{i}\cos \theta )_{i}F^{ii}u_{ii}\right )\\
&-(C|Du|F+C|Du|)\\
=&T_{1}+T_{2}+T_{3}.
\end{aligned}
\end{equation}

For the term $T_1$, according to \eqref{uiiestimate} we have
\begin{equation}
\begin{aligned}
T_{1}=&\sum\limits_{i\in K}\left((\frac{1}{v}-\frac{u_{i}^{2}}{v^{3}})F^{ii}u_{ii}^{2}-2( d_{i}\cos \theta)_{i}F^{ii}u_{ii}\right )\\
\geq&\sum\limits_{i\in K}(-2(d_{i}\cos \theta)_{i}F^{ii}u_{ii})\geq-Cv^{2}F\,,
\end{aligned}
\end{equation}
and for the term $T_2$, because of the definition of $K$ and the fact $ax^2+bx\ge -\frac{b^2}{4a}$ for $a>0$, we have
\begin{equation}
\begin{aligned}
T_{2}\geq&\sum\limits_{i\notin K}\left((\frac{1}{v}-\frac{u_{i}^{2}}{v^{3}})F^{ii}u_{ii}^{2}-2(d_{i}\cos \theta)_{i}F^{ii}u_{ii}\right )\\
\geq & \sum\limits_{i\notin K}(\frac{C}{vF^{ii}}(F^{ii}u_{ii})^{2}-2(d_{i}\cos \theta)_{i}F^{ii}u_{ii})\geq-CvF\,.
\end{aligned}
\end{equation}

It follows that
\begin{equation}
\begin{aligned}
I=\frac{F^{ij}w_{ij}}{w}\geq-CvF-CF-C\,.
\end{aligned}
\end{equation}

For the term $II$,
\begin{equation}
\begin{aligned}
II= &[h^{''}-(h^{'})^{2}]F^{ij}u_{i}u_{j}=[h^{''}-(h^{'})^{2}]\sum\limits_{i=1}^{n}F^{ii}u_{i}^{2}\\
\geq & [h^{''}-(h^{'})^{2}]\sum\limits_{i\in K}F^{ii}u_{i}^{2}\geq C[h^{''}-(h^{'})^{2}]v^{2}F\,.
\end{aligned}
\end{equation}

By the Newton-MacLaurin inequality stated in Proposition \ref{NewtonInequality}, we have
\begin{equation}
\begin{aligned}
F\geq C>0,
\end{aligned}
\end{equation}
therefore,
\begin{equation}
\begin{aligned}
0&\geq\frac{F^{ij}\Phi_{ij}}{F}=I+II+III\geq C[h^{''}-(h^{'})^{2}]v^{2}-Cv-C-\frac{C}{F}\\
&\geq C[h^{''}-(h^{'})^{2}]v^{2}-Cv-C\,.
\end{aligned}
\end{equation}

If we take ~$h(t)=\frac{1}{2}\ln \frac{1}{{(3M - t)}}$~, then~$h^{''}-(h^{'})^{2}=(h^{'})^{2}$ and $h(t)$ satisfies all the assumptions we have set in advance.  Thus we bound the gradient at this point such that~$v\leq C$~, then we derive the gradient estimate near the boundary by a standard discussion. Thus we complete the proof of Theorem \ref{PCAB}

\end{proof}
\section{Oblique derivative boundary value}

 In this section, we will get the a priori gradient estimate of the solution to Hessian equations with oblique derivative boundary value. Specifically, we will show the following result.

\begin{theorem}\label{OBB}
Let $\Omega$ be a smooth bounded domain in $\mathbf{R}^n(n\ge 2)$ and $u$ be the admissible solution to the following Hessian equations with oblique derivative boundary value,
\begin{equation}\label{ODB}
   \begin{split}
   \left\{\aligned
&\sigma_{k}(u_{ij})=f(x,\,u)  &\mathrm{\mbox {in}} & \ \Omega,\\
&\frac{\partial u}{\partial\beta}=\varphi(x,\,u) &\mathrm{\mbox {on}}& \   \partial\Omega.\\
\endaligned\right.
   \end{split}
\end{equation}
Where $f(x,\,t)$ is a positive smooth function defined on $\Omega \times \mathbf{R}$ with $f_t \ge 0$, $\varphi(x,\,t)$ is a smooth function defined on $\overline{\Omega}\times\mathbf{R}$ and $\beta$ is a smooth unit vector field along $\partial \Omega$ with $ \langle \beta,\ \nu \rangle\geq c_0>0$ for some positive constant $c_0$,  $\nu$~ is denoted to be the inward unit normal along ~$\partial\Omega$~. Also we assume that we have already got the ~$C^{0}$~ estimate as ~$|u|\leq M$~. Then, there exists a positive constant $C=C(M,\, n,\, \Omega,\, c_0,\, |\beta|_{C^3({\partial\Omega})} ,\, |f|_{C^1(\Omega\times [-M,\ M])},\, |\varphi|_{C^3(\Omega\times [-M,\ M])})$ such that
\begin{equation}
\begin{aligned}
|Du|\le C.
\end{aligned}
\end{equation}

\end{theorem}

\begin{proof}
Firstly, we say some words about the boundary value.

Taking a unit normal moving frame along $\partial \Omega$, denonted by $ \{e_1,\,e_2,\,\cdots, \,e_{n-1},\, \nu  \}$, then $\beta$ can be represented as
\begin{equation}
\begin{aligned}
\beta=\beta_n \nu+\sum_{l=1}^{n-1} \beta_l e_l\,,
\end{aligned}
\end{equation}
where $\beta_n=\langle \beta,\ \nu\rangle=\cos \theta$ which is bounded from below by the positive constant $c_0$ according to the conditions of Theorem \ref{OBB}.

By the boundary data, we have
\begin{equation}\label{b1}
\begin{aligned}
\varphi (x,\,u) = \frac{{\partial u}}{{\partial \beta }} =  < Du, \beta  >  = \frac{{\partial u}}{{\partial \nu}}{\beta _n} + \sum\limits_{l = 1}^{n - 1} {{\beta _l}{u_l}}\,.
\end{aligned}
\end{equation}

Setting $w = u - \frac{{\varphi d}}{{\cos \theta }}$, we then have
\begin{equation}\label{b2}
\begin{aligned}
\varphi (x,\,u)= \frac{{\partial (w + \frac{{\varphi d}}{{\cos \theta }})}}{{\partial \nu}}\cos \theta  + \sum\limits_{l = 1}^{n - 1} {{\beta _l}{(w + \frac{{\varphi d}}{{\cos \theta }})_l}}\,,
\end{aligned}
\end{equation}
which indicates that
\begin{equation}\label{b3}
\begin{aligned}
0 = \frac{{\partial w}}{{\partial \nu}}{\beta _n} + \sum\limits_{l = 1}^{n - 1} {{\beta _l}{w_l}}\,.
\end{aligned}
\end{equation}
Therefore, we have
\begin{equation}\label{b4}
\begin{aligned}
\frac{{\partial w}}{{\partial \nu}} =  - \sum\limits_{l = 1}^{n - 1} {\frac{{{\beta _l}}}{{{\beta _n}}}{w_l}}
\end{aligned}
\end{equation}
and it follows by Cauchy inequality and the fact $\sum\limits_{i=1}^n\beta_i^2=1$ that
\begin{equation}\label{b-over}
\begin{aligned}
{(\frac{{\partial w}}{{\partial \nu}})^2} \le |Dw{|^2} \cdot {\sin ^2}\theta.
\end{aligned}
\end{equation}

As before, we only need to get the gradient estimate near boundary, denoted by ~$\Omega_{\mu}$~, where ~$\mu\le \mu_1$~ is a positive constant to be determined later. We extend $\beta$ smoothly to $\Omega_\mu$, also denoted by $\beta$, such that $\langle \beta,\ Dd\rangle=\cos \theta\ge c_0$ is also assumed to be still valid. Denote by

\[\phi  = |Dw{|^2} - {(\sum\limits_{\alpha  = 1}^n {{w_\alpha }{d_\alpha }} )^2} = \sum\limits_{\alpha ,\delta  = 1}^n {({\delta _{\alpha \delta }} - {d_\alpha }{d_\delta }){w_\alpha }{w_\delta }}  = \sum\limits_{\alpha ,\delta  = 1}^n {{C^{\alpha \delta }}{w_\alpha }{w_\delta }} \]
and take the auxiliary function
$$\Phi  = \log \phi  + h(u) + \tau d, $$
where $h(t)$ is a smooth function, $\tau$ is a positive constant. Both of them will be determined later.

Assume the maximum of $\Phi$ on $\Omega_{\mu}$ is achieved at $x_0$. Also by the interior gradient estimate which has been derived in \cite{XJW1}, we only need to consider the two following cases.

{\bf Case I:~$x_{0}\in\partial\Omega$.}

As in Section 3, we choose a coordinate around $x_0$ such that $\nu=\frac{\partial}{\partial x_n}$, and $\frac{\partial}{\partial x_i}(i=1,2,\cdots,n-1)$ are tangent to $\partial \Omega$. We also have that

$$\frac{{\partial d}}{{\partial {x_i}}} = 0,\ \  \frac{{\partial d}}{{\partial {x_n}}} = 1,\ \  \frac{{{\partial ^2}d}}{{\partial {x_n}\partial {x_\alpha }}} = 0,\ \  \frac{{{\partial ^2}d}}{{\partial {x_i}\partial {x_j}}} =  - {\kappa _i}{\delta _{ij}}\,,$$
 where $1\le i,j<n-1$, $1\le \alpha \le n$ and $\kappa_i(i=1,2,\cdots,n-1)$ are the principal curvatures of $\partial \Omega$ at $x_0\in \partial \Omega$.

By the fact that $x_0$ is the maximum point of $\Phi$ on the boundary, it follows that
   \begin{equation}\label{bPhi-i}
\begin{split}
   0 =& \Phi_i = \frac{\phi_i}{\phi} + h'u_i, \ \ i=1,2,\cdots,n-1\,,
     \end{split}
\end{equation}
  and
  \begin{equation}\label{bPhi-n}
\begin{split}
   0 \ge& \Phi_n = \frac{\phi_n}{\phi} + h'u_n + \alpha d_n = \frac{w_n}{w} + h'u_n + \alpha\,.
     \end{split}
\end{equation}

From \eqref{bPhi-i}, we get
\begin{equation}\label{bPhi-i-1}
\begin{split}
   - \phi h'{u_i} = {(|Dw|^2)_i} - [(\sum\limits_{\alpha=1}^{n}{w_\alpha }{d_\alpha })^2]_i = 2\sum\limits_{j=1}^{n-1}{w_{ij}}{w_j}  - 2{w_n}\sum\limits_{j=1}^{n-1}{d_{ij}}{w_j}, \ \ i=1,2,\cdots,n-1\,.
     \end{split}
\end{equation}

We then deal with the term $\phi_n$ as follows.
\begin{equation}\label{bphi-n-1}
\begin{split}
  {\phi _n} =& 2\sum\limits_{\alpha  = 1}^n {{w_\alpha }{w_{an}}}  - 2{w_n}{w_{nn}} = 2\sum\limits_{i = 1}^{n - 1} {{w_i}{w_{in}}}  = 2\sum\limits_{i = 1}^{n - 1} {{w_i}{w_{ni}}}  + 2\sum\limits_{i,j = 1}^{n - 1} {{\kappa _{ij}}{w_i}{w_j}} \\
 =&  - 2\sum\limits_{i = 1}^{n - 1} {{w_i}{{(\frac{{{\beta _l}}}{{{\beta _n}}}{w_l})}_i}}  + 2\sum\limits_{i,j = 1}^{n - 1} {{\kappa _{ij}}{w_i}{w_j}} \\
  =&  - \frac{{2\sum\limits_{i,l = 1}^{n - 1} {{w_i}{w_{li}}{\beta _l}} }}{{{\beta _n}}} - 2\sum\limits_{i,l = 1}^{n - 1} {{w_i}{w_l}{{(\frac{{{\beta _l}}}{{{\beta _n}}})}_i}}  + 2\sum\limits_{i,j = 1}^{n - 1} {{\kappa _{ij}}{w_i}{w_j}} \\
 =& \frac{{\phi h'\sum\limits_{l = 1}^{n - 1} {{u_l}{\beta _l}} }}{{{\beta _n}}} - \frac{{2{w_n} \sum\limits_{l,j = 1}^{n - 1} {{d_{lj}}{w_j}{\beta _l}}  }}{{{\beta _n}}} - 2\sum\limits_{i,l = 1}^{n - 1} {{w_i}{w_l}{{(\frac{{{\beta _l}}}{{{\beta _n}}})}_i}}  + 2\sum\limits_{i,j = 1}^{n - 1} {{\kappa _{ij}}{w_i}{w_j}}
     \end{split}
\end{equation}
Note that the last equality comes from \eqref{bPhi-i-1} and we denote by  $k_{ij}$ the Weingarten matrix of the boundary with respect to $\nu$\,.

Therefore, it follows that
\begin{equation}\label{bphi-n-over}
\begin{split}
  0 \ge {\Phi _n} =& \frac{{\frac{{\phi h'\sum\limits_{l = 1}^{n - 1} {{u_l}{\beta _l}} }}{{{\beta _n}}} - \frac{{2{w_n}\sum\limits_{j,l = 1}^{n - 1} {{d_{lj}}{w_j}{\beta _l}} }}{{{\beta _n}}} - 2\sum\limits_{i,l = 1}^{n - 1} {{w_i}{w_l}{{(\frac{{{\beta _l}}}{{{\beta _n}}})}_i}}  + 2\sum\limits_{i,j = 1}^{n - 1} {{\kappa _{ij}}{w_i}{w_j}} }}{\phi } + h'{u_n} + \tau  \\
 =& \frac{{ - \frac{{2{w_n} \sum\limits_{l,j = 1}^{n - 1} {{d_{lj}}{w_j}{\beta _l}}  }}{{{\beta _n}}} - 2\sum\limits_{i,j = 1}^{n - 1} {{w_i}{w_l}{{(\frac{{{\beta _l}}}{{{\beta _n}}})}_i}}  + 2\sum\limits_{i,j = 1}^{n - 1} {{\kappa _{ij}}{w_i}{w_j}} }}{\phi } +\frac{h'\varphi}{\cos \theta } + \tau\,.
     \end{split}
\end{equation}
We may assume in advance that
\begin{equation}\label{hsuppose1}
  \begin{split}
0<h'(t)<1,\ \  \forall \, t\in [-M,\ M]\,.
  \end{split}
\end{equation}
 Thus, if we set $\tau$ large enough, depending upon $c_0,\, |\beta|_{C^1(\partial\Omega)},\, n$ and the geometry of $\partial\Omega$, we can conclude that this case does not occur at all.

{\bf Case II:~$x_{0}\in\Omega_{\mu}$\,.}

All the calculations will proceed at this point and the Einstein summation convention will be adopted during all the calculations if no otherwise specified.  Also, we denoted by $F^{ij}$ the derivative $\frac{\partial \sigma_{k}(u_{ij})}{\partial u_{ij}}$ and  $F$ the sum $\sum_{i=1}^n F^{ii}.$

According to \cite{XJW1}, we know that
\begin{equation}\label{to-boundary}
\begin{split}
  {\sup _\Omega }|D u| \le {C_1}(1 + {\sup _{\partial \Omega }}|Du|),
     \end{split}
\end{equation}
where $C_1$ is a positive constant depending only on $\Omega,\,n,\,k,\, |D_xf|_{C^0(\Omega\times[-M,\ M])}$. One can verify this point by setting a auxiliary function $\chi= \log |Du|^2 + \alpha |x|^2$ and checking that $F^{ij}\chi_{ij}\ge 0$ once we set $\alpha$ to be small and  $|Du|$ to be large enough. Remark that we have supposed with out loss of generality that the point $0$ is located out of $\overline{\Omega}$.

Now we assume that the maximum value of $|Du|$ on $\partial \Omega$ is achieved at the point $x_1$, without loss of generality, we can suppose that
\begin{equation}\label{bBound}
  \begin{split}
  |Du|^2({x_1}) \ge 4{\sup _{\partial \Omega }}{({\rm{|}}\frac{\varphi }{{\cos \theta }}{\rm{|}})^2},
  \end{split}
\end{equation}
otherwise we have finish the estimate of the gradient of the solutions.

By the fact that $\Phi(x_0)\ge \Phi(x_1)$, it follows that
\begin{equation}\label{phix0}
  \begin{split}
\phi ({x_0}) \ge & C(\tau ,\mu ){{\rm{e}}^{ - 2Mh'}}\phi ({x_1}) = C(\tau ,\mu ){{\rm{e}}^{ - 2Mh'}}[|Dw{|^2} - {(\frac{{\partial w}}{{\partial \nu}})^2}]({x_1})\\
\ge & C(\tau ,\mu ){{\rm{e}}^{ - 2Mh'}}[|Dw{|^2}{\cos ^2}\theta ]({x_1})\\
 \ge & {c_0}^2C(\tau ,\mu ){{\rm{e}}^{ - 2Mh'}}|Dw{|^2}({x_1})\\
 = &{c_0}^2C(\tau ,\mu ){{\rm{e}}^{ - 2Mh'}}|D u - \frac{\phi }{{\cos \theta }}\nu{|^2}({x_1})\\
 \ge & \frac{{{c_0}^2C(\tau ,\mu ){{\rm{e}}^{ - 2Mh'}}}}{4}|D u{|^2}({x_1}),
  \end{split}
\end{equation}
remark that the last inequality above comes from \eqref{bBound} and the fact that $(x-y)^2\ge \frac{x^2}{2}-y^2$.

Joining with \eqref{to-boundary} and assuming once again that
\begin{equation}\label{hsuppose2}
  \begin{split}
0<h'(t)<\frac{1}{2M},\ \  \forall \, t\in [-M,\ M],
  \end{split}
\end{equation}
we then derive
\begin{equation}\label{phix0-over}
  \begin{split}
\phi ({x_0})
 \ge & \frac{{{c_0}C(\tau ,\mu ){{\rm{e}}^{ - 2Mh'}}}}{{4{C_1}}}(\mathop {\sup }\limits_\Omega  |Du{|^2} - {C_1})\\
  \ge &  \frac{{{c_0}C(\tau ,\mu ){{\rm{e}}^{ - 2Mh'}}}}{{8{C_1}}}\mathop {\sup }\limits_\Omega  |Du{|^2}\\
 \ge &  \frac{{{c_0}C(\tau ,\mu ){{\rm{e}}^{ - 2Mh'}}}}{{8{C_1}}}|Du{|^2}({x_0})\\
 \ge &  \frac{{{c_0}C(\tau ,\mu )}}{{9{C_1}\rm{e}}}|Dw{|^2}({x_0})\triangleq C_0 |Dw{|^2}({x_0}) .
  \end{split}
\end{equation}
Without loss of generality, we can assume that $C_0\in (0,\ 1).$

At $x_0$, we also follow \cite{DM} to choose the coordinate such that $(u_{ij})$ is diagonal.

For $k=1,\,2,\cdots,n$, denote by ${T_k} = \sum\limits_{l = 1}^n {{C^{kl}}{w_l}} $ and $\overrightarrow{T}=(T_1,\,T_2,\cdots,T_n)$, it is obvious to observe that $|\overrightarrow{T}| \le |Dw|$ and
\begin{equation}\label{phiT}
  \begin{split}
\phi  = \sum\limits_{i,j = 1}^n {{C^{ij}}{w_i}{w_j}}  = \sum\limits_{j = 1}^n {{T_j}{w_j}}=\langle\overrightarrow{T},\ Dw \rangle.
  \end{split}
\end{equation}

Considering the lower bound we just derived in \eqref{phix0-over}, we get
\begin{equation}\label{TDw}
  \begin{split}
C_0 |Dw| \le {\rm{|}}T{\rm{|}} \le |Dw|.
  \end{split}
\end{equation}

Without loss of generality, we further assume by the Pigeon-Hole Principle that
\begin{equation}\label{T1W1-inequality}
  \begin{split}
{T_1}{w_1} \ge \frac{C_0 }{n}|Dw{|^2},
  \end{split}
\end{equation}
therefore,
\begin{equation}\label{T1W1-inequality-over}
  \begin{split}
\frac{{{w_1}}}{{{T_1}}} \ge \frac{C_0}{n}
  \end{split}
\end{equation}
and we can set $\mu$ is small such that
\begin{equation}\label{T1u1-inequality-over}
  \begin{split}
\frac{{{u_1}}}{{{T_1}}} \ge \frac{C_0}{3n}.
  \end{split}
\end{equation}

By a direct calculation, we have
\begin{equation}\label{wij-uij}
  \begin{split}
{w_i} =& {u_i}(1 - \frac{{{\varphi _z}d}}{{\cos \theta }}) + \varphi_i(\frac{d}{\cos\theta})+ \varphi(\frac{d}{\cos\theta})_i\,;\\
{w_{ij}} =& {u_{ij}}(1 - \frac{{{\varphi _z}d}}{{\cos \theta }}) - \frac{{{\varphi _{zz}}d}}{{\cos \theta }}{u_i}{u_j} -\varphi_{zj}u_i(\frac{d}{\cos\theta})- {\varphi _z}{u_i}{(\frac{d}{{\cos \theta }})_j} + {\varphi _z}{u_j}{(\frac{d}{{\cos \theta }})_i}\\
  &+ \frac{d}{{\cos \theta }}{\varphi _{ij}} +\varphi_{zi}u_j(\frac{d}{\cos\theta})+ {(\frac{d}{{\cos \theta }})_i}{\varphi _j} + {(\frac{d}{{\cos \theta }})_j}{\varphi _i} + {(\frac{d}{{\cos \theta }})_{ij}}\varphi\,.
  \end{split}
\end{equation}

By the assumption that $x_0$ is the maximum point, we then have $\Phi_i =0$ for $i=1,\,2,\,\cdots,\,n$, it follows that
\begin{equation}\label{1-ordercondition}
  \begin{split}
\frac{\phi_i}{\phi} + h'{u_i}+\tau {d_i}= 0\,,
  \end{split}
\end{equation}
especially for $i=1$,
\begin{equation}\label{Phii0}
  \begin{split}
\sum\limits_{l = 1}^n {{T_l}{w_{l1}}}  =  - \frac{\phi }{2}(h'{u_1} + \tau {d_1}) - \sum\limits_{k,l = 1}^n {\frac{{{C^{kl}}_{,1}}}{2}{w_k}{w_l} }\,,
  \end{split}
\end{equation}
then by \eqref{TDw}$-$ \eqref{T1u1-inequality-over}, we have
\begin{equation}\label{u11}
  \begin{split}
   {u_{11}}(1 - \frac{{{\varphi _z}d}}{{\cos \theta }}) \le - \frac{{{u_1}}}{{2{T_1}}}h'\phi + Cd|Dw{|^2} + C|Dw|\,.
  \end{split}
\end{equation}
If we assume that $h'$ has a positive lower bound and $|Dw|$ is large enough, and $\mu$ is small enough, then we can get
\begin{equation}\label{u11-over}
  \begin{split}
   {u_{11}}<0\,,
  \end{split}
\end{equation}
thus
\begin{equation}\label{F11}
  \begin{split}
{F^{11}} \ge {F^{kk}} \ge C\left( {n,k} \right)F\,.
  \end{split}
\end{equation}

Now, it is turn for us to deal with the second order derivatives of $\Phi$. With the help of the first order condition \eqref{1-ordercondition}, it follows that
\begin{equation}\label{Phi-ij}
  \begin{split}
{\Phi _{ij}} = &\frac{{{{\left( {\sum\limits_{k,l = 1}^n {{C^{kl}}{w_k}{w_l}} } \right)}_{ij}}}}{\phi } - (h'{u_i} + \tau {d_i})(h'{u_j} + \tau {d_j}) + h'{u_{ij}} + h''{u_i}{u_j} + \tau {d_{ij}}\\
 = & \frac{{{{\left( {\sum\limits_{k,l = 1}^n {{C^{kl}}{w_k}{w_l}} } \right)}_{ij}}}}{\phi } - h'{u_i}{d_j} - h'{u_j}{d_i} - {\tau ^2}{d_i}{d_j} + h'{u_{ij}} + [h'' - {(h')^2}]{u_i}{u_j} + \tau {d_{ij}}.
  \end{split}
\end{equation}
Hence, we have at $x_0$ that
\begin{equation}\label{FijPhiij}
  \begin{split}
0\ge {F^{ij}}{\Phi _{ij}} = & \frac{{{F^{ij}}{{\left( {\sum\limits_{k,l = 1}^n {{C^{kl}}{w_k}{w_l}} } \right)}_{ij}}}}{\phi } - 2h'\sum\limits_{i,j = 1}^n {{F^{ij}}{u_i}{d_j}}  - {\tau ^2}\sum\limits_{i,j = 1}^n {{F^{ij}}{d_i}{d_j}}  + h'kf\\
 &+ [h'' - {(h')^2}]\sum\limits_{i,j = 1}^n {{F^{ij}}{u_i}{u_j}}  + \tau \sum\limits_{i,j = 1}^n {{F^{ij}}{d_{ij}}} \\
 \ge& \sum\limits_{i,j = 1}^n {\frac{{{F^{ij}}{{\left( {{C^{kl}}{w_k}{w_l}} \right)}_{ij}}}}{\phi }}  - ({\tau ^2} + 1)\sum\limits_{i,j = 1}^n {{F^{ij}}{d_i}{d_j}}  + h'kf \\
 &+ [h'' - 2{(h')^2}]\sum\limits_{i,j = 1}^n {{F^{ij}}{u_i}{u_j}}  + \tau \sum\limits_{i,j = 1}^n {{F^{ij}}{d_{ij}}} \\
 =&  I + II + III + IV + V.
  \end{split}
\end{equation}

It is a simple and direct calculation to deal with the last four terms. According to \eqref{TDw}$-$ \eqref{T1u1-inequality-over} and \eqref{F11}, we have
\begin{equation}\label{2-4}
  \begin{split}
&II =  - ({\tau ^2} + 1){F^{ij}}{d_i}{d_j} \ge  - ({\tau ^2} + 1)F\,,\\
&III = h'kf \ge 0\,,\\
&IV= [h'' - 2{(h')^2}]\sum\limits_{i,j = 1}^n {{F^{ij}}{u_i}{u_j}}  \ge [h'' - 2{(h')^2}]{F^{11}}{u_1}^2  \ge C_2[h'' - 2{(h')^2}]{\rm{|}}Dw{{\rm{|}}^2}F\,,\\
&V = \tau {F^{ij}}{d_{ij}} \ge  - {k_0}\tau \sum\limits_{i = 1}^n {{F^{ii}}}  =  - {k_0}\tau F\,.
  \end{split}
\end{equation}
where $k_0$ is a positive constant related to the geometry of $\partial\Omega$.

To deal with the term $I$, we have
\begin{equation}\label{bI}
  \begin{split}
I =& \frac{{\sum\limits_{i,j,k,l = 1}^n {{F^{ij}}{C^{kl}}_{,ij}{w_k}{w_l}} }}{\phi }+ \frac{{2\sum\limits_{i,j,k,l = 1}^n {{F^{ij}}{C^{kl}}{w_{ijk}}{w_l}} }}{\phi }\\
 & + \frac{{4\sum\limits_{i,j,k,l = 1}^n {{F^{ij}}{C^{kl}}_{,j}{w_{ik}}{w_l}} }}{\phi }+ \frac{2{\sum\limits_{i,j,k,l = 1}^n {{F^{ij}}{C^{kl}}{w_{ik}}{w_{jl}}} }}{\phi }\\
 =& {I_1} + {I_2} + {I_3} + {I_4}.
  \end{split}
\end{equation}
 We consider these four terms one by one in the following text.

 For the term $I_1$, it is easy to deduce that
 \begin{equation}\label{I1}
  \begin{split}
{I_1} = \frac{{\sum\limits_{i,j,k,l = 1}^n {{F^{ij}}{C^{kl}}_{,ij}{w_k}{w_l}} }}{\phi } \ge  - CF\,.
  \end{split}
\end{equation}

For the term $I_2$, we need a subtle operation as follows.
\begin{equation}\label{I2}
  \begin{split}
\phi {I_2} = & 2\sum\limits_{i,j,k,l = 1}^n {{F^{ij}}{C^{kl}}{w_{ijl}}{w_k}}  = 2\sum\limits_{i,j,l = 1}^n {{F^{ij}}{T_l}{{(u - \frac{{\varphi (x,u)d}}{{\cos \theta }})_{ijl}}}} \\
 = & 2\sum\limits_{i,j,l = 1}^n {{T_l}\left[ {{F^{ij}}{u_{ijl}} + {F^{ij}}{{(\frac{{\varphi (x,u)d}}{{\cos \theta }})_{ijl}}}} \right]}\\
  = & 2\sum\limits_{i,j,l = 1}^n {{T_l}\left[ {f_l + {F^{ij}}{{(\frac{{\varphi (x,u)d}}{{\cos \theta }})_{ijl}}}} \right]}\,.
  \end{split}
\end{equation}
To proceed, we should compute $(\frac{{\varphi (x,u)d}}{{\cos \theta }})_{ijl}$\,. By a direct calculation,
\begin{equation}\label{ijl}
  \begin{split}
(\frac{{\varphi (x,u)d}}{{\cos \theta }})_{ijl}=& {(\varphi )_{ijl}}(\frac{d}{{\cos \theta }}) + {(\varphi )_{ij}}{(\frac{d}{{\cos \theta }})_l} + {(\varphi )_{il}}{(\frac{d}{{\cos \theta }})_j} + {(\varphi )_{jl}}{(\frac{d}{{\cos \theta }})_i} \\
&+ {(\varphi )_i}{(\frac{d}{{\cos \theta }})_{jl}} + {(\varphi )_j}{(\frac{d}{{\cos \theta }})_{il}} + {(\varphi )_l}{(\frac{d}{{\cos \theta }})_{ij}} + \varphi {(\frac{d}{{\cos \theta }})_{ijl}}\,,
  \end{split}
\end{equation}
where
\begin{equation}\label{varphi-derivative}
  \begin{split}
{(\varphi )_i} =& {\varphi _i} + {\varphi _z}{u_i}\,,\\
{(\varphi )_{ij}} =& {\left( {{\varphi _i} + {\varphi _z}{u_i}} \right)_j} = {\varphi _{ij}} + {\varphi _{iz}}{u_j} + {\varphi _{zj}}{u_i} + {\varphi _{zz}}{u_i}{u_j} + {\varphi _z}{u_{ij}}\,,\\
{(\varphi )_{ijl}} =& {\left( {{\varphi _{ij}} + {\varphi _{iz}}{u_j} + {\varphi _{zj}}{u_i} + {\varphi _{zz}}{u_i}{u_j} + {\varphi _z}{u_{ij}}} \right)_l}\\
=&{\varphi _{ijl}} + {\varphi _{ijz}}{u_l} + {\varphi _{izl}}{u_j} + {\varphi _{izz}}{u_j}{u_l} + {\varphi _{iz}}{u_{lj}}\\
 &+ {\varphi _{zjl}}{u_i} + {\varphi _{zjz}}{u_i}{u_l} + {\varphi _{zj}}{u_{il}}\\
 &+ {\varphi _{zzl}}{u_i}{u_j} + {\varphi _{zzz}}{u_i}{u_j}{u_l} + {\varphi _{zz}}{u_{li}}{u_j} + {\varphi _{zz}}{u_i}{u_{lj}}\\
 &+ {\varphi _{zl}}{u_{ij}} + {\varphi _{zz}}{u_l}{u_{ij}} + {\varphi _z}{u_{ijl}}\,.
  \end{split}
\end{equation}
Note that
$$\sum\limits_{i,j = 1}^n {{F^{ij}}{u_{ij}}}  = kf,\ \sum\limits_{i,j = 1}^n {{F^{ij}}{u_{ijl}} }  =D_lf, \  \sum\limits_{j = 1}^n {{F^{ij}}{u_{lj}}}  = {F^{ii}}{u_{ii}}\ (\mbox{fixed}\ i),\  0 < \sum\limits_{j = 1}^n {{F^{ij}}{u_i}{u_j}}  \le |Du{|^2}F,$$
therefore we have
\begin{equation}\label{I2-over}
  \begin{split}
\phi {I_2} \ge  - Cd|Dw{|^4}F - C|Dw{|^3}F - Cd|Dw{|^2}\sum\limits_{i = 1}^n {|{F^{ii}}{u_{ii}}} | - C{\rm{|}}Dw{\rm{|}}\sum\limits_{i = 1}^n {|{F^{ii}}{u_{ii}}} |- C|Dw{|^2}\,.
  \end{split}
\end{equation}

Almost the same procedure, we can settle the remained two terms.
\begin{equation}\label{I3}
  \begin{split}
\phi {I_3} =  &4\sum\limits_{i,j,p,l = 1}^n {{F^{ij}}{C^{pl}}_{,j}{w_{ip}}{w_l}}  \ge  2\sum\limits_{i,l = 1}^n {{w_l}{C^{il}}_{,i}{F^{ii}}{u_{ii}}}  -C |Dw{|^3}F\\
 \ge &  - C|Dw{|^3}F - C{\rm{|}}Dw{\rm{|}}\sum\limits_{i = 1}^n {|{F^{ii}}{u_{ii}}} |\,,
  \end{split}
\end{equation}
and
\begin{equation}\label{I4}
  \begin{split}
\phi {I_4} = &2 \sum\limits_{i,j,p,l = 1}^n {{F^{ij}}{C^{pl}}{w_{ip}}{w_{jl}}} \\
\ge& 2 \sum\limits_{i = 1}^n {{F^{ii}}{C^{ii}}{u_{ii}}^2}  - Cd|Dw{|^2}\sum\limits_{i = 1}^n {|{F^{ii}}{u_{ii}}} | - Cd|Dw{|^4}F - C|Dw{|^3}F\,.
  \end{split}
\end{equation}

Taking into account \eqref{I1}, \eqref{I2-over}, \eqref{I3} and \eqref{I4}, we can get
\begin{equation}\label{I-over}
  \begin{split}
\phi I\ge & 2\sum\limits_{i = 1}^n {{F^{ii}}{C^{ii}}{u_{ii}}^2}  - Cd|Dw{|^2}\sum\limits_{i = 1}^n {|{F^{ii}}{u_{ii}}} |- C{\rm{|}}Dw{\rm{|}}\sum\limits_{i = 1}^n {|{F^{ii}}{u_{ii}}} |\\
& - Cd|Dw{|^4}F - C|Dw{|^3}F- C|Dw{|^2}\,.
  \end{split}
\end{equation}

Denoting by
\begin{equation}\label{H}
  \begin{split}
H=2\sum\limits_{i = 1}^n {{F^{ii}}{C^{ii}}{u_{ii}}^2}  - Cd|Dw{|^2}\sum\limits_{i = 1}^n {|{F^{ii}}{u_{ii}}} |- C{\rm{|}}Dw{\rm{|}}\sum\limits_{i = 1}^n {|{F^{ii}}{u_{ii}}} |\,,
  \end{split}
\end{equation}
and we will bound $H$ from below in the following.

Plugging
\begin{equation}\label{F11u11}
  \begin{split}
{F^{11}}{u_{11}} = kf - \sum\limits_{\alpha  = 2}^n {{F^{\alpha \alpha }}{u_{\alpha \alpha }}}
  \end{split}
\end{equation}
into \eqref{H}, we then have
\begin{equation}\label{H1}
  \begin{split}
H=&2\sum\limits_{\alpha = 2}^n {{F^{\alpha\alpha}}{C^{\alpha\alpha}}{u_{\alpha\alpha}}^2} - \left(Cd|Dw{|^2}+ C{\rm{|}}Dw{\rm{|}}   \right)\sum\limits_{\alpha = 2}^n {|{F^{\alpha\alpha}}{u_{\alpha\alpha}}} |\\
&+2{{F^{11}}{C^{11}}{u_{11}}^2}- \left(Cd|Dw{|^2}+ C{\rm{|}}Dw{\rm{|}}   \right){|{F^{11}}{u_{11}}} |\\
\ge & \frac{{2{C^{11}}}}{{{F^{11}}}}{\left( {kf - \sum\limits_{\alpha  = 2}^n {{F^{\alpha \alpha }}{u_{\alpha \alpha }}} } \right)^2} + 2\sum\limits_{\alpha  = 2}^n {{F^{\alpha \alpha }}{C^{\alpha \alpha }}{u_{\alpha \alpha }}^2} \\
 &- \left[ {Cd|Dw{|^2} + C|Dw|} \right]\sum\limits_{\alpha  = 2}^n {{\rm{|}}{F^{\alpha \alpha }}{u_{\alpha \alpha }}|}  - C|Dw{|^2}\\
 = &\frac{{{2C^{11}}}}{{{F^{11}}}}{\left( {\sum\limits_{\alpha  = 2}^n {{F^{\alpha \alpha }}{u_{\alpha \alpha }}} } \right)^2} + \frac{{{C^{11}}}}{{{F^{11}}}}{\left( {kf} \right)^2} - \frac{{2kf{C^{11}}}}{{{F^{11}}}}\sum\limits_{\alpha  = 2}^n {{\rm{|}}{F^{\alpha \alpha }}{u_{\alpha \alpha }}{\rm{|}}} \\
 &+ \sum\limits_{\alpha  = 2}^n {\frac{{{2C^{\alpha \alpha }}}}{{{F^{\alpha \alpha }}}}{{\left( {{F^{\alpha \alpha }}{u_{\alpha \alpha }}} \right)}^2}} - \left[ {Cd|Dw{|^2} + C|Dw|} \right]\sum\limits_{\alpha  = 2}^n {{\rm{|}}{F^{\alpha \alpha }}{u_{\alpha \alpha }}|}  - C|Dw{|^2}\,,
  \end{split}
\end{equation}
By the Newton-MacLaurin inequality stated in Proposition \ref{NewtonInequality}, it follows $F\geq C>0$,
thus joining with \eqref{F11} we then have
\begin{equation}\label{H2}
  \begin{split}
H\ge & \frac{{{2C^{11}}}}{{{F^{11}}}}{\left( {\sum\limits_{\alpha  = 2}^n {{F^{\alpha \alpha }}{u_{\alpha \alpha }}} } \right)^2} + \sum\limits_{\alpha  = 2}^n {\frac{{{2C^{\alpha \alpha }}}}{{{F^{\alpha \alpha }}}}{{\left( {{F^{\alpha \alpha }}{u_{\alpha \alpha }}} \right)}^2}} \\
 & - \left[ {Cd|Dw{|^2} + C|Dw|} \right]\sum\limits_{\alpha  = 2}^n {{\rm{|}}{F^{\alpha \alpha }}{u_{\alpha \alpha }}|}  - C|Dw{|^2}\,.
  \end{split}
\end{equation}

To reach our target, we set out to consider the following quadratic form with respect to $x_\alpha,\ \alpha=2,\,3,\,\cdots,\, n$\,,
\begin{equation}\label{QFJ}
  \begin{split}
J=\frac{{{C^{11}}}}{{{F^{11}}}}{\left( {\sum\limits_{\alpha  = 2}^n {{x_\alpha }} } \right)^2} + \sum\limits_{\alpha  = 2}^n {\frac{{{C^{\alpha \alpha }}}}{{{F^{\alpha \alpha }}}}{x_\alpha }^2}\,,
  \end{split}
\end{equation}
it is obvious to get
\begin{equation}\label{QFG1}
  \begin{split}
J \ge \frac{{{C^{11}}{{\left( {\sum\limits_{\alpha  = 2}^n {{x_\alpha }} } \right)}^2} + \sum\limits_{\alpha  = 2}^n {{C^{\alpha \alpha }}{x_\alpha }^2} }}{F} \,.
  \end{split}
\end{equation}
We then consider the simpler quadratic form
\begin{equation}\label{QFG2}
  \begin{split}
{\rm{G(}}{C^{11}},{C^{22}}, \cdots ,{C^{nn}},{x_2}, \cdots ,{x_n}{\rm{) = }}{C^{11}}{\left( {\sum\limits_{\alpha  = 2}^n {{x_\alpha }} } \right)^2} + \sum\limits_{\alpha  = 2}^n {{C^{\alpha \alpha }}{x_\alpha }^2} \,.
  \end{split}
\end{equation}

Since
\[0 \le {C^{ii}} \le 1,\  \sum\limits_{i = 1}^n {{C^{ii}}}  = n - 1\,,\]
at most one of all the $C^{ii}$'s is permitted to be equal to zero, therefore the quadratic form $G$ with fixed coefficients $C^{ii}$'s is positive definite. Now, we can consider $G$ as a $2n-1$ variables positive function and restrict $G$ on the compact domain
\[{\rm{D = \{ (}}{C^{11}},{C^{22}}, \cdots ,{C^{nn}},{x_2}, \cdots ,{x_n}{\rm{)}}{\kern 1pt} {\kern 1pt} {\kern 1pt} |\ \ 0 \le {C^{ii}} \le 1,\ 1 \le i \le n,\  \sum\limits_{i = 1}^n {{C^{ii}}}  = n - 1,\  \sum\limits_{\alpha  = 2}^n {{x_\alpha }^2} {\rm{ = }}1{\rm{\} }}.\]
It is a simple fact that the minimum value of $G$ on $\mathrm{D}$ is the least eigenvalue of the $G$, we assume it to be $\lambda_0 $ which is a fixed positive constant. Hence, we have
\begin{equation}\label{G-over}
  \begin{split}
{\rm{G}} \ge \lambda_0 \sum\limits_{\alpha  = 2}^n {{x_\alpha }^2}.
  \end{split}
\end{equation}

Thus, we have by the simple fact $ax^2+bx\ge -\frac{b^2}{4a}$ if $a>0$ that
\begin{equation}\label{H3}
  \begin{split}
H \ge & \frac{2 \lambda_0}{F} \sum\limits_{\alpha  = 2}^n (F^{\alpha\alpha}u_{\alpha\alpha})^2 - \left[ {Cd|Dw{|^2} + C|Dw|} \right]\sum\limits_{\alpha  = 2}^n {{\rm{|}}{F^{\alpha \alpha }}{u_{\alpha \alpha }}|}  - C|Dw{|^2}\\
\ge &- \frac{Cd}{\lambda _0}|Dw{|^4}F - C|Dw{|^3}F - C|Dw{|^2} \,.
  \end{split}
\end{equation}
Plugging this into \eqref{I-over} and joining with \eqref{phix0-over}, we can derive
\begin{equation}\label{I-finished}
  \begin{split}
 I \ge  - Cd|Dw{|^2}F - C|Dw|F\,.
  \end{split}
\end{equation}

Therefore, combining \eqref{2-4} and \eqref{I-finished} we can get
\begin{equation}\label{Last}
  \begin{split}
0\ge \frac{{F^{ij}}{\Phi _{ij}} }{F}\ge C_2[h'' - 2{(h')^2}]{\rm{|}}Dw{{\rm{|}}^2}-  C_3d|Dw{|^2} - C|Dw|\,,
  \end{split}
\end{equation}
where we use once again the fact $F\ge C>0$.

Now, we set
\begin{equation}\label{Ht}
  \begin{split}
h(t) = \frac{1}{4}\ln \frac{1}{{(3M - t)}}\,,
  \end{split}
\end{equation}
it satisfies all the assumptions we have made in advance. Let $\mu$ be small enough so that $C_3\mu \le C_2(h')^2$, we then get
\begin{equation}\label{Last-over}
  \begin{split}
 C_2{\left( {\frac{1}{{16{M}}}} \right)^2}{\rm{|}}Dw{{\rm{|}}^2} - C|Dw|\le 0,
  \end{split}
\end{equation}
this will lead to the universal bound of $|Dw|$ at $x_0$ and we then get the global gradient estimate of $u$ on $\overline{\Omega}$ by a standard discussion and this finishes the whole proof of Theorem \ref{OBB}.

\end{proof}

{\bf Acknowledgments:} The research belongs to the project ZR2020MA018 supported by Shandong  Provincial Natural Science Foundation. The author would like to owe thanks to Prof. X. Ma for his constant encouragement and useful discussion on this topic.

\end{document}